\documentclass[letterpaper,conference,10pt]{ieeeconf}
\usepackage{verbatim}
\usepackage{booktabs}
\usepackage{soul}
\usepackage{dsfont}
\usepackage{enumerate}
\usepackage{amsmath,amssymb}
\usepackage{subfigure}
\usepackage{stfloats}
\usepackage{todonotes}
\usepackage[hidelinks]{hyperref}
\usepackage{graphicx}
\usepackage{epstopdf}
\usepackage{tikz}
\usepackage{lipsum}
\usepackage{circuitikz}
\usetikzlibrary{decorations.text}
\usepackage{pgfplots}
\usepgfplotslibrary{polar}
\pgfplotsset{compat=newest}
\usepackage{cite}
\usepackage{stfloats}
\usepackage{MnSymbol}
\usepackage{lipsum}
\usepackage{amsthm}
\newtheorem{defn}{Definition}

\newtheorem{lem}[defn]{Lemma}

\newtheorem{assum}[defn]{Assumption}

\newtheorem{thm}[defn]{Theorem}

\providecommand{\R}{\ensuremath \mathbb{R}}
\providecommand{\N}{\ensuremath \mathbb{N}}
\providecommand{\ip}[1]{\ensuremath \langle #1\rangle}
\providecommand{\spt}{\ensuremath \text{spt}}

\newtheorem{remark}{Remark}

\tikzstyle{stuff_nofill}=[rectangle,draw,font={A}]
\tikzstyle{stuff_fill}=[rectangle,draw=none,fill=white,font={A}]

\title{Synthesizing the Optimal Luenberger-type Observer for Nonlinear Systems}

\author{Shankar~Mohan,~Jinsun~Liu~and~Ram~Vasudevan
\thanks{Shankar Mohan and Jinsun Liu are with the Department of Electrical Engineering and Computer Science, University of Michigan, Ann Arbor, MI 48109 \texttt{\{elemsn,jinsunl\}@umich.edu}}
 \thanks{Ram Vasudevan is with the Department of Mechanical Engineering, University of Michigan, Ann Arbor, MI 48109 \texttt{ramv@umich.edu}}
}

\IEEEoverridecommandlockouts                              

\usepackage[T1]{fontenc}

\begin{document}
\maketitle
\begin{abstract}
Observer design typically requires the observability of the underlying system, which may be hard to verify for nonlinear systems, while guaranteeing asymptotic convergence of errors, which may be insufficient in order to satisfy performance conditions in finite time. 
This paper develops a method to design Luenberger-type observers for nonlinear systems which guarantee the largest possible domain of attraction for the state estimation error regardless of the initialization of the system.
The observer design procedure is posed as a two step problem. 
In the the first step, the error dynamics are abstractly represented as a linear equation on the space of Radon measures. 
Thereafter, the problem of identifying the largest set of initial errors that can be driven to within the user-specified error target set in finite-time for all possible initial states, and the corresponding observer gains, is formulated as an infinite-dimensional linear program on measures. 
This optimization problem is solved, using Lasserre's relaxations via a sequence of semidefinite programs with vanishing conservatism.
By post-processing the solution of step one, the set of gains that maximize the size of tolerable initial errors is identified in step two.
To demonstrate the feasibility of the presented approach two examples are presented.

\end{abstract}

\section{Introduction}

Estimating the state of a system is critical to a variety of control related tasks including feedback design, diagnostics, and monitoring.
Unfortunately, measuring the states of a system can involve considerable engineering effort and cost, and sometimes may be impossible. 
Observers serve to provide a means to achieve this objective and their design has been an active area of interest within the controls community. 

The observers considered in the literature usually satisfy several requirements. 
First, they are typically designed to guarantee asymptotically stable error dynamics though finite-time convergence of the estimation error may be more useful in certain applications \cite{perruquetti2008finite,asad2000slid,du2013recursive,frye2010fast}. 
This has inspired several recent papers that have, inspired by the ideas proposed by Luenberger, proposed observers for linear and nonlinear systems \cite{perruquetti2008finite,engel2002continuous,menold2003finite,allgower1999nonlinear}. 

Second, the systems for which observers are designed are usually presumed to satisfy some observability condition.
In the case of nonlinear systems, this often requires assuming that the system is transformable to the observer canonical form and sometimes even requires that the system be observable along the solution trajectory, or at the origin \cite{xia1989nonlinear,krener1985nonlinear,kazantzis1998nonlinear}. 
This condition can be difficult to check in practice and thus results in an \textit{ad hoc} application of the observer design technique.
\par

Third, observers are typically designed to have globally convergent error dynamics.
Though this is an attractive property, construncting such observers may not be possible for every system.
In practice, an observer that is locally convergent is generally sufficient as long as one can explicitly describe the neighborhood of initial observer states that converge to the true state of the system regardless of its true initial state.

The primary contribution of this paper is the development of an automated tool to design observers for nonlinear systems that rely on static output injection to reject disturbances and uncertainty.
These numerically synthesized observers satisfy the following characteristics: (1) they do not presume the observability of the system or the existence of a transformation that renders the system observable; (2) they guarantee the finite-time behavior of the error dynamics; and, (3) they find the largest possible domain of attraction for the error dynamics.
In addition to specifying the dynamics of the system, to utilize the approach presented in this paper, a user must specify the state space of the system, the time for which the system will evolve, and the error state that they wish the dynamics to converge to within the pre-specified time.
The result of the technique presented in this paper is a static output injection gain and the largest set of initial observer states that are provably able to converge to the user specified error state in the specified time for all states of the system initialized in the user-specified state space.

The presented approach relies on dividing the observer design problem into two sub-problems.
The first sub-problem identifies the largest set of static gains for output injection and associated initial states for the observer that are able to be driven to a user specified error in finite time for all initial states in the state space.
The second sub-problem utilizes this result to identify a single (or a set) static gain with the largest set of initial observer states that are convergent. 

To tractably solve each sub-problem, this paper first transforms the dynamics of the nonlinear system and observer into a linear system over the space of measures \cite{folland2013real}. 
As a result, each sub-problem can be posed as an infinite dimensional linear program over the space of measures.
In the instance of polynomial or rational dynamics, the solution to this infinite dimensional linear program can be found with vanishing conservatism using a hierarchy of semidefinite programs. 
This solution methodology is inspired by several recent papers \cite{henrion2014convex,holmes2016convex,mohan2016convex}.

The remainder of the paper is organized as follows:
Section~\ref{sec:preliminaries} introduces the notation used in the remainder of the paper. 
Section~\ref{sec:prob} formulates the first sub-problem as an infinite-dimensional linear program on measures and describes a sequence of the semidefinite programs with vanishing conservatism to solve the first sub-problem;
Section~\ref{sec:findL} presents a method to solve the second sub-problem;
Section~\ref{sec:examples} demonstrates the performance of the presented approach on examples; and Section~\ref{sec:conclusion} concludes the paper.

\section{Preliminaries}
\label{sec:preliminaries}

This section describes the class of systems under consideration, form of the observer that is constructed, and outlines the problem of interest.

\subsection{Notation}

The following notation is adopted in the remainder of the text.
Sets are italicized and capitalized.
The set of continuous functions on a compact set $K$ are denoted by $\mathcal C(K)$.
The ring of polynomials in $x$ is denoted by $\R[x]$, and the degree of a polynomial is equal to the degree of its largest multinomial; the degree of the multinomial $x^\alpha,\,\alpha\in \N_{\ge 0}^n$ is $|\alpha|=\|\alpha\|_1$; and $\R_d[x]$ is the set of polynomials in $x$ with maximum degree $d$.
The dual to $\mathcal C(K)$ is the set of Radon measures on $K$, denoted as $\mathcal M(K)$, and the pairing of $\mu\in \mathcal M(K)$ and $v\in \mathcal C(K)$ is denoted:
  \begin{align}
  \ip{\mu,v}=\int_{K}v(x)\,d\mu(x).
  \end{align}
We denote the non-negative Radon measures by ${\cal M}_+(K)$. 
The space of Radon probability measures on $K$ is denoted by ${\cal P}(K)$. 
If a measure $\nu\in \mathcal M_+(A\times B)$ can be represented as a product measure of $\eta\in M_+(A)$ and $\zeta\in M_+(B)$, we write $\nu = \eta\otimes \zeta$. 
The Lebesgue measure on a set $A$ is denoted by $\lambda_A$. 
The support of a measure, $\mu$, is identified as $\spt(\mu)$. 
For convenience, the interval $[0,T]$ is denoted by $\mathcal T$, when necessary.


\subsection{Problem Formulation}

We next formally describe the problem of interest. 
In this paper we consider drift systems with observations of the following form:
\begin{align}
\begin{aligned}
    \dot{x}(t) &= f(t,x(t)) \\
    y(t) &= h(x(t))
    \end{aligned}
    \label{eq:sys}
\end{align}
where $x(t) \in X \subset \R^n$ are the states of the system and $C: \R^n \to Y$ describes a linear transformation from the state to the output, $y \in Y \subset \R^m$. 
Note that though the dynamics of this system may be known, usually the initial condition of this system is unknown and may start anywhere in $X$.
As a result, we construct an observer of the form:
\begin{align}
\begin{aligned}
\dot {\hat x}(t) = &\,f(t,\hat x(t))+l(y(t)-\hat y(t)) := \tilde f(t,x(t),\hat x(t),l) \\
\hat y(t) = &\,h(\hat x(t)),
\end{aligned}
\label{eq:obs}
\end{align}
where $\hat{x}(t) \in \hat{X} \subset \R^n$ and $l \in L \subset \R^{n\times m}$ is a constant gain that we design.
The objective of this paper is to find a gain $l$ in \eqref{eq:obs} that results in the largest possible set of initial observer states converging satisfactorily close to the true state of the system in a finite amount of time, $T$, regardless of the true initial state of the system.

To describe this objective explicitly, we first define the state estimation error, $e(t):=x(t)-\hat x(t) \in E$, dynamics as:
\begin{align}
\begin{aligned}
    \dot e(t) =& f(t,x(t)) - \tilde f(t,x(t),x(t)-e(t),l) \\
            :=& g(t,x(t),e(t),l).
\end{aligned}
\end{align}
Next, we define the augmented system as follows:
\begin{align}
  \begin{bmatrix}
  \dot x(t)\\
  \dot {e}(t)
  \end{bmatrix} = 
  \begin{bmatrix}
  f(t,x(t))\\
   g(t,x(t),e(t),l)
  \end{bmatrix}.
\end{align}
With $z(t) := \begin{bmatrix} x(t) & e(t) \end{bmatrix}' \in Z$ where $Z := X \times E$, the above equation can be written as:
\begin{align}
  \dot z(t) = \phi(t,z(t),l).
  \label{eq:aug}
\end{align}
In addition, let $E_T \subset \R^n$ correspond to a target state that the user wishes to drive the estimation error into by time $T$ and let $Z_T := X \times E_T$ be the target set in the augmented state space.
To formally state the objective of this paper, we next define the set of gains and associated initial error states that can be driven to $E_T$ by time $T$ under the augmented dynamics for all possible initial states of the system in Equation \eqref{eq:sys}:
\begin{align}
	\label{eq:ROA}
		{\cal X} = \Big\{ (e_0,l) &\in E \times L \mid \forall x_0 \in X ~ \exists z:[0,T] \to X \times E \nonumber \\ 
		&\textrm{ s.t } \dot{z}(t) = \phi\big(t,z(t),l\big) \text{ a.e. }t \in [0,T]  \nonumber \\ &~ z(0) = (x_0,e_0),~z(T) \in Z_T\Big\}.
\end{align}
We refer to this set as the \emph{backwards reachable set of $E_T$}.
Given this definition, the objective of this paper is to compute a gain as:
\begin{flalign}
   & &  \sup_{l \in L } & \phantom{3} \lambda_E\big( \{ e_0 \mid (e_0,l) \in {\cal X} \} \big)          && \label{eq:optgain} 
\end{flalign}
In words, this optimization problem seeks to find the gain which drives the largest set of initial error states to the desired error target set for all possible initial states of the system in Equation \eqref{eq:sys}.

Our approach to solving this problem mirrors our problem formulation. 
That is, we first compute the backwards reachable set of $E_T$ and then solve the optimization problem in Equation \eqref{eq:optgain} to find an optimal gain.
To ensure that the problem is well-posed, we make the following assumptions:
\begin{assum}
$f$ is Lipschitz in $x$ and piecewise continuous in $t$.
\end{assum}
\begin{assum}
$X$, $E$, $E_T$, and $L$ are compact subsets.  
\end{assum}

\subsection{Occupation Measures}

This section describes how to compute the backwards reachable set by transforming the nonlinear dynamics of the system into the space of measures. 
The result of this transformation is a linear description of the dynamics. 
To formulate this transformation, this section introduces occupation measure (refer to \cite{henrion2014convex,Mohan2016} for more details).

Given an initial condition for the system in Equation~\eqref{eq:aug}, $z_0$, the \emph{occupation measure} quantifies the amount of time spent by an evaluated trajectory in any subset of the space.
The occupation measure $\mu(\cdot\mid z_0,l)$ is defined as:
\begin{align}
  \mu(A\times B\times C\mid z_0,l) = \int_{0}^T I_{A\times B\times C}(t,z,l\mid z_0,l)\,dt,
  \label{eq:occ}
\end{align}
for all Borel Sets $A \times B \times C \subset {\cal T} \times Z \times L$ where $I_{K}(y)$ is the indicator function on the set $K$ that returns one if $y\in K$ and zero otherwise.
With the above definition of the occupation measure, one can show:
\begin{align}
  \ip{\mu(\cdot\mid z_0,l),v} = \ip{\lambda_{\cal T},v(t,z(t\mid z_0,l),l)},
  \label{eq:ip}
\end{align}
for all $v \in {\cal C}({\cal T} \times Z \times L)$.

The occupation measure completely characterizes the solution trajectory of the system resulting from an initial condition.
Since we are interested in the collective behavior of a set of initial conditions, we define the {\em average} occupation measure as:
\begin{align}
\mu(A\times B\times C) = \int_{Z\times L} \mu(A\times B\times C\mid z_0,l)\,d\bar \mu_0,
\end{align}
where $\bar \mu_0 \in  \mathcal M_+( Z\times L)$ is the un-normalized probability distribution of initial conditions.
The {\em average} occupation measure of a set in $\mathcal T\times Z\times L$ is equal to the cumulative time spent by all solution trajectories that begin in $\spt(\bar \mu_0)$.

By applying the Fundamental Theorem of Calculus, one can evaluate a test function $v\in \mathcal C^1(\mathcal T\times Z\times L)$ at time $t=T$ along a solution to Equation~\eqref{eq:aug} as:
\begin{align}
  v(T,z(T\mid z_0,l),l) =&\, v(0,z_0,l) + \hspace*{-0.2cm} \int_0^T \hspace*{-0.35cm} \mathcal L_{\phi}v\big(t,z(t \mid z_0,l ),l\big) dt,
  \label{eq:ftc}
\end{align}
where $\mathcal L_{\phi}:  {\mathcal C}^1(\mathcal T\times Z\times L) \to  {\mathcal C}(\mathcal T\times Z\times L)$ is defined as:
\begin{align}
  \mathcal L_{\phi}v := \frac{\partial v}{\partial z}\cdot \phi+\frac{\partial v}{\partial t},
\end{align}
Using Equation~\eqref{eq:ip}, Equation~\eqref{eq:ftc} can be re-written as:
\begin{align}
\begin{split}
  v(T,z(T\mid z_0,l),l) =&\, v(0,z_0,l) + \\
  +&\,\int\limits_{Z\times L}\mathcal L_{\phi}v\,d\mu(t,z,l\mid z_0,l).
\end{split}
\label{eq:ftc2}
\end{align}
Integrating Equation~\eqref{eq:ftc2} with respect to $\bar \mu_0$ and defining a new measure $\mu_T\in \mathcal M_+ (Z_T\times L)$, as:
\begin{align}
  \mu_T(A\times B) = \int\limits_{Z\times L}I_{A\times B}(x(T\mid z_0,l),l)\,d\bar \mu_0,
\end{align}
produces the following equality:
\begin{align}
  \ip{\delta_T\otimes \mu_T,v}= \ip{\delta_0\otimes\bar \mu_0,v}+\ip{\mu,\mathcal L_{\phi}v},
  \label{eq:lv}
\end{align}
where, with a slight abuse of notations, $\delta_t$ is used to denote a Dirac measure situated at time $t$.
Using adjoint notation, Equation~\eqref{eq:lv} can be written as:
\begin{align}
  \delta_T\otimes \mu_T= \delta_0\otimes\bar \mu_0+\mathcal L_{\phi}'\mu.
  \label{eq:adjointlv}
\end{align}
Equation~\eqref{eq:adjointlv} is a version of Liouville's Equation, holds for all test function $v\in \mathcal C^1(\mathcal T\times Z\times L)$, and summarizes the visitation information of all trajectories that emanate from $\spt(\bar \mu_0)$ and terminate in $\spt(\mu_T)$.
\section{Computing Feasible Observer Gains}
\label{sec:prob}

This section describes how to formulate and solve for the backwards reachable set defined in Equation \eqref{eq:ROA} using the occupation measures defined in the previous section.
In particular, our approach relies upon describing the evolution of the augmented system in Equation \eqref{eq:aug} using a family of measures $(\bar \mu_0,\mu_T,\mu)$ which satisfy Equation \eqref{eq:lv} while optimizing for the $\bar \mu_0$ with the largest possible support.
As we describe below this translates into an infinite dimensional linear program over measures. 

Recall that computing the backwards reachable set defined in Equation \eqref{eq:ROA} requires finding observer gains and associated initializations for the observer state that ensure all possible initial states of the system are satisfactorily estimated (i.e. the estimation error converges to $E_T$ by time $T$). 
In particular, note that the choice of the gain and the initial state of the observer cannot depend on the \emph{true} state of the system since that is not known \textit{a priori}.
This implies that the initial values of the error and gain state should be independent of the values of the initial system state.
That is, $\bar \mu_0$ is expressible as a product measure of the form $\bar \mu_0 = \mu_0\otimes \lambda_X$, where $\mu_0 \in {\cal M}_+( E \times L )$.




\subsection{Convex Computation of Feasible Gains}

The computation of ${\cal X}$ can be be posed as the solution to an infinite dimensional Linear Program (LP) on measures:
\begin{flalign}\nonumber
    & & \sup_{\Lambda} \hspace*{1cm} & \ip{\mu_{0},\mathds 1_{E \times L}} && (P)\\
    & & \text{st.} \hspace*{1cm} &\lambda_X\otimes \mu_{0}+\mathcal L_{\phi}'\mu=\,\mu_{T}, && \label{eq:primal:liouville}\\
    & & & \mu_{0}+\hat\mu_{0}=\,\lambda_{E \times L}, \label{eq:primal:slack}
\end{flalign}
where $\Lambda:=(\mu_0,\hat\mu_{0},\mu,\mu_T) \in \mathcal M_+(E\times L)\times \mathcal M_+(E\times L) \times {\mathcal M}_+({\mathcal T} \times Z \times L) \times \mathcal M_+( Z_{T}\times L )$ and $\mathds 1_{E\times L}$ denotes the function that takes value $1$ everywhere on $E \times L$.
The following property of $(P)$ can be derived using \cite[Lemma~1, Theorem~1]{henrion2014convex}:
\begin{lem}
\label{lem:primal:equiv}
    Let $p^*$ be the optimal value of $(P)$, then $p^* = \lambda_{E \times L}( {\cal X })$. 
    Moreover, the supremum is attained with the $\mu_0$-component of the optimal solution equal to the restriction of the Lebesgue measure to the backwards reachable set ${\cal X}$.

\end{lem}


\noindent The dual problem to $(P)$ is \cite{Anderson1987}:
\begin{flalign*}
  & &  \inf_{\Xi}&\phantom{3} \ip{\lambda_{el},w}          && (D) \\
  & &  \text{st.}&\phantom{3}\mathcal L_{\phi} v(t,z,l)\le 0         && \forall (t,z,l)\in \mathcal T\times Z\times L\\
  & &  &\phantom{3} w(e,l) \ge 0                                         && \forall (e,l)\in E\times L\\
  & &  &\phantom{3} w - \ip{\lambda_X,v(0,z,l)} - 1 \ge 0                     && \forall (e,l)\in E\times L\\
  & & & \phantom{3} v(T,z,l)\ge 0                                        && \forall (z,l)\in Z_T\times L
\end{flalign*}
where $\Xi := (v,w)\in \mathcal C^1(\mathcal T\times Z\times L)\times \mathcal C(E\times L)$. 
We use this dual representation of the problem to identify the support of the $\mu_0$-component of the optimal solution of $(P)$.
To do this, we first establish the equivalence between $(P)$ and $(D)$ using \cite[Theorem~2]{henrion2014convex}:
\begin{lem}
  There is no duality gap between problems $(P)$ and $(D)$.
\end{lem}

Feasible pairs to $(D)$ have an interesting interpretation: $v$ is similar to a Lyapunov function for the system, and $w$ resembles an indicator function on $\spt(\mu_0)$, which follows from the Fundamental Theorem of Calculus and the constraints of $(D)$:
\begin{lem}
\label{lem:dual:level-set}
	Let $(v,w)$ be a pair of feasible functions to $(D)$. 
	The 1-super level set of $w$ contains $\spt(\mu_0)$.
\end{lem}
As a result, the $1$-super level set of the $w$-component of any feasible pair of functions to $D$ is an outer approximation to ${\cal X}$.
In fact, one can prove that the solution to to $(D)$ coincides with ${\cal X}$ by using \cite[Theorem~3]{henrion2014convex}
\begin{thm}
\label{thm:indicator}
 There is a sequence of feasible solutions to $(D)$ whose $w$ component converges uniformly in the $L^1$ norm to the indicator function on ${\cal X}$.
\end{thm}

\subsection{Solving $(P)$ via Semidefinite Programming}

Problem $(P)$ is an infinite dimensional linear program on measures, which is usually impossible to solve exactly.
This section introduces a convex relaxation hierarchy whose solutions converge with vanishing conservatism to the true solution to $(P)$. 
This sequence of relaxations is constructed by characterizing each measure using a sequence of moments\footnote{The $n$th moment of a measure $\mu$ is
  $y_{\mu,n}=\ip{\mu,x^n}.$}
and assuming the following:
\begin{assum}
$f$ is a polynomial function and $X,E,E_T,$ and $L$ are semi-algebraic sets.
  \label{assump:poly}
\end{assum}
\noindent We also make the following assumption on the semi-algebraic sets to ensure that we can construct a Semidefinite Programming (SDP) hierarchy (refer to \cite[Theorem~2.15]{lasserre2009moments}):
\begin{assum}
Each of the semi-algebraic sets $X,E,E_T,$ and $L$ has at least one defining polynomial of the form $R - \|x\|_2^2$ for some constant $R \geq 0$.
\end{assum}
\noindent This assumption is made without loss of generality since $X,E,E_T,$ and $L$ is bounded and therefore this redundant constraint can be added for a sufficiently large constant.


Under these assumption, given any finite $d$-degree truncation of the moment sequence of all measures in $(P)$, a relaxation, $(P_d)$, can be formulated over the moments of measures to construct a SDP.
The dual to $(P_d)$, $(D_d)$, can be expressed as a Sums-of-Squares (SOS) program by considering $d$-degree polynomials in place of the continuous variables in $D$.
In the interest of brevity of presentation, only $(D_d)$ is presented below. This decision is motivated by the fact that solution to $(D_d)$ can be used to identify the $\spt(\mu_0)$.

To formalize this dual program, first note that a polynomial $p \in \R[x]$ is SOS or $p \in \text{SOS}$ if it can be written as $p(x) = \sum_{i=1}^m q_i^2(x)$ for a set of polynomials $\{q_i\}_{i=1}^m \subset \R[x]$.
Note that efficient tools exist to check whether a finite dimensional polynomial is SOS using SDPs~\cite{parrilo2000structured}.
To formulate this problem, we make a few additional definitions.
Suppose we are given a semi-algebraic set $A = \{x \in \R^n \mid h_{i}(x) \geq 0, h_i \in \R[x], \forall i \in \N_m \}$, then define the $d$-degree {\em quadratic module} of $A$ as:
\begin{align}
  \begin{split}
  Q_d(A)=\bigg\{q\in \R_d[x]\,\bigg|\, \exists \{s_k\}_{k \in \{0,1,...,m\} \cup \{0\}} \subset \text{SOS s.t. } \\
  q=s_0+\sum_{k\in \{1,...,m\}}h_{k}s_k \bigg\}
  \end{split}
\end{align}

With this definition, the $d$-degree relaxation of the dual, $D_d$, can be written as:
  \begin{flalign}\nonumber
    & & \inf_{\Xi_d} \hspace*{0.1cm} & \int_{E\times L}w_d(e,l)\,d(\lambda_{el}) && \hspace*{-0.3cm} (D_d) \nonumber\\
    & & \text{st.} \hspace*{0.1cm} & w_d\in Q_d(E\times L) && \\
    & & & v_d(T,z,l)\in Q_d(Z_T \times L) && \\
    & & & -\mathcal L_{\phi}v_d(t,z,l)\in Q_d(\mathcal T\times Z \times L) && \\
    & & & w_d-\ip{\lambda_x,v_d(0,z,l)}-1\in Q_d(E\times L)
  \end{flalign}
where $\Xi_d=\Big\{ \big(v_d,w_d\big) \in \R_d[t,z,l]\times \R_d[e,l]\Big\}$.
The solution to this SDP can be used to generate an outer approximation to ${\cal X}$ which converges to ${\cal X}$ as the relaxation degree increases:
\begin{lem}\cite[Theorem~6]{henrion2014convex}
\label{lem:implementation:outer}
Let $w_d$ denote the $w$-component of the solution to $(D_d)$, then ${\cal X}_{d} = \{(e_0,l) \in E\times L \mid w_d(e,l) \geq 1 \}$ is an outer approximation of ${\cal X}$ and $\lim_{d\to\infty}\lambda_{E \times L}({\cal X}_{d} \backslash {\cal X} ) = 0$.
\end{lem}

\begin{remark}
The method described in this section generates an outer approximation of $\mathcal{X}$. 
In fact a similar approach can be used to derive an inner approximation of ${\cal X}$ \cite{korda2013inner}.
\end{remark}

\section{Choosing the optimal gain}
\label{sec:findL}
This section presents a pair of methods to address the problem presented in Equation~\eqref{eq:optgain} once ${\cal X}$ is computed using the methods presented in the previous section.
As a result of Theorem \ref{thm:indicator}, one can use the optimal $w$-component of the solution to $(D)$ to rewrite Equation~\eqref{eq:optgain} as:
\begin{flalign}
l^*\in \arg\max_{l\in L} \int_E w^*(e,l)\,de
\label{eq:findl:main}
\end{flalign}
However, $(D)$ cannot be solved directly. 
In this section, we describe a method to utilize the $w$-component of $(D_d)$ to identify the set of \emph{optimal} gains.

The first step in using the $d$-degree optimal solution, $w_d^*$, in Equation~\eqref{eq:findl:main} is to recognize that for each $l \in L$:
\begin{flalign}
&&&\int\limits_E w_d(e,l)\,de \ge \hspace*{-0.25in} \int\limits_{\{e\mid w_d(e,l)\ge 1\}} \hspace*{-0.25in} w_d(e,l)\,de \ge \int\limits_E w^*(e,l)\,de.
\label{eq:findl:ineq}
\end{flalign}
As a result of Lemma~\ref{lem:implementation:outer}, one can prove that the above inequalities turn into equalities as $d\rightarrow \infty$. 
For finite $d$, however, it is necessary to approximate the last term in the above equation, as tightly from above or below, as possible. 
In this section, we present a state-space discretization to evaluate the set-integration as defined by the second term in Equation~\eqref{eq:findl:ineq} efficiently, and also approximately solve the resulting version of the problem in Equation~\eqref{eq:findl:main}. 


Define $\beta:L\rightarrow \R_{\ge 0}$ as:
 \begin{align}
   \beta(l) = \left(\prod_j^n \Delta e_j \right) \sum_{i_1=1}^{N_1}\cdots \sum_{i_n=1}^{N_n} \left(\min\{1,w_d(e_{\mathcal I},l)\}\right)^k
   \label{eq:beta}
 \end{align}
 where $e_{\mathcal I} := e_{i_1,\ldots,i_n}$ is a point in the n-cuboid discretization in each dimension, $k\gg 0$, and $\Delta e_j$ is the width of the uniform grid in the $j$th coordinate.
As proven next, the function $\beta$ converges uniformly to the evaluation of the cost of the optimization problem in Equation~\eqref{eq:findl:main}, with $w^*$ replaced with $w_d$:
\begin{lem}
For each $l \in L$, $\beta$ converges uniformly from above to $\int_E w^*(e,l)\,de$ as $d \to \infty$.
\end{lem}
\begin{proof}
The proof is similar to that of \cite[Lemma 12]{holmes2016convex}.
\end{proof}
To choose the optimal gain, we fix a degree relaxation and discretize the space of gains $L$ and evaluate $\beta$ for each discretization.
By maximizing $\beta$, one can select the optimal gain.

\section{Examples}
\label{sec:examples}

This section provides three $2$D numerical experiments. 
Each SDP is prepared using a custom software toolbox and the modeling tool YALMIP \cite{lofberg2005yalmip}.
The programs are run with the commercial solver MOSEK on a machine with 144 64-bit 2.40GHz Intel Xeon CPUs and 1 Terabyte memory.
The end time in each example is set as $T=1$, and the observer gain $l$ is restricted to $L:=\{l\in\mathbb{R}^2 \mid 10-\|L\|_2\geq0\}$. 
The error space is assumed to be $ E:=\{e\in\mathbb{R}^2 \mid 1-\|e\|_2\geq0\}$, and the target error set is $E_T:=\{e\in\mathbb{R}^2 \mid 0.05-\|e\|_2\geq0\}$.
A degree $6$ relaxation is used to solve the examples. 
For simplicity, we say an observer gain $l$ is admissible given initial condition $e_0$, if the estimation error is driven into $E_T$ at $t=T$ by $l$ for all the initial condition $x_0\in X$.
Similarly an initial condition $e_0$ is feasible given observer gain $l$, if the estimation error is driven into $E_T$ at $t=T$ by $l$ for all the initial condition $x_0\in X$.
While computing $\beta$ as described in Equation \eqref{eq:beta}, $k$ is set equal to $1000$.

\subsection{2D Linear System}
\label{subsec:linear}
To validate the performance of our numerical method, we begin by considering a two dimensional linear system:
\begin{align}
\dot{x}_1 & = -x_1-3x_2\\
\dot{x}_2 & = -2x_1-6x_2\\
y &= x_1
\end{align}
where $x\in X:=\{x\in\mathbb{R}^2 \mid 1-\|x\|_2\geq 0 \}$. 
$w_d$ is first computed using $(D_d)$ and then $\beta$ was computed using $w_d$ as depicted in Figure \ref{fig:eg3}. 
The optimal gains according to the method proposed in this paper are all points belonging to the gray region in Figure \ref{fig:eg3}. 

To verify the correctness of this computed region, the gain space was uniformly sampled in polar coordinates with $2601$ points.
If all sampled initial errors in $E$ could be driven to $E_T$ for all sampled initial states, then this point was depicted in black in Figure \ref{fig:eg3}.
These black points are the sampled ground truth optimal gains. 
Notice that the gray region which we compute using our proposed method is an outer approximation to the set of ground truth optimal gains.

\begin{figure}[!h]
\centering  
 {\includegraphics[trim =0in 0in 0in  0in,width=\columnwidth,clip=true]{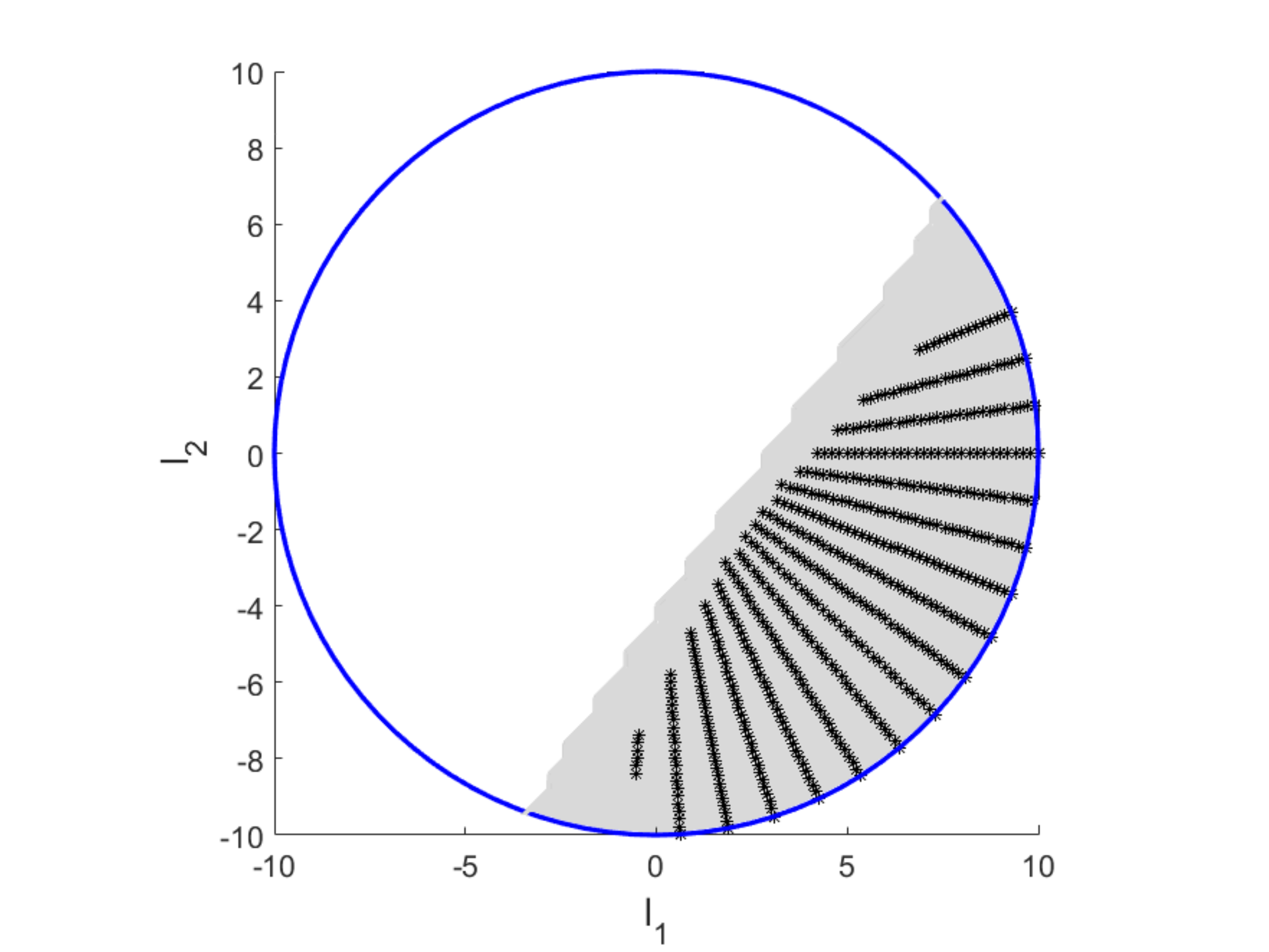}}
\caption{An illustration of the computed optimal gains for static observer design (the gray region) and the sampled ground truth optimal observer gains (black dots) as described in Section \ref{subsec:linear}.}
\label{fig:eg3}
\end{figure}

\subsection{2D Nonlinear System}
\label{subsec:nonlinear}

Consider the following 2-dimensional nonlinear system:
\begin{align}
\dot{x}_1 & = -x_1+x_1x_2\\
\dot{x}_2 & = -x_2\\
y &= x_1
\end{align}
where $x\in X:=\{x\in\mathbb{R}^2|1-\|x\|_2\geq0\}$. 
The ground truth admissible $l$ is generated by sampling the entire space of $L$ with $1200$ points under the uniform distribution in polar coordinates. 
By varying the initial condition $e_0$, the admissible area of observer gains changes as shown in Figure \ref{fig:eg2}. 
This means that there does not exist an $l$ that works for all $e_0$. 

However, an optimal statical observer gain $l_{\textrm{opt}}$ can be obtained from $w_d$ based on Section \ref{sec:findL}, such that $l_{\textrm{opt}}$ works for the largest set of initial errors $e_0\in E$.
Figure \ref{fig:eg2_hist} compares the performance of this computed optimal gain to the best gain, $l_{\textrm{sample}}$ we could find via sampling the entire state space with $961$ points using a uniform distribution in polar coordinates and another arbitrary gain in $L$.
The number of feasible $e_0$ for our computed $l_{\textrm{opt}}$ is only two less than the number for $l_{\textrm{sample}}$ and it is significantly better than the arbitrary selected gain.

\begin{figure}[!h]
\centering
\subfigure{\includegraphics[width=\columnwidth,clip=true]{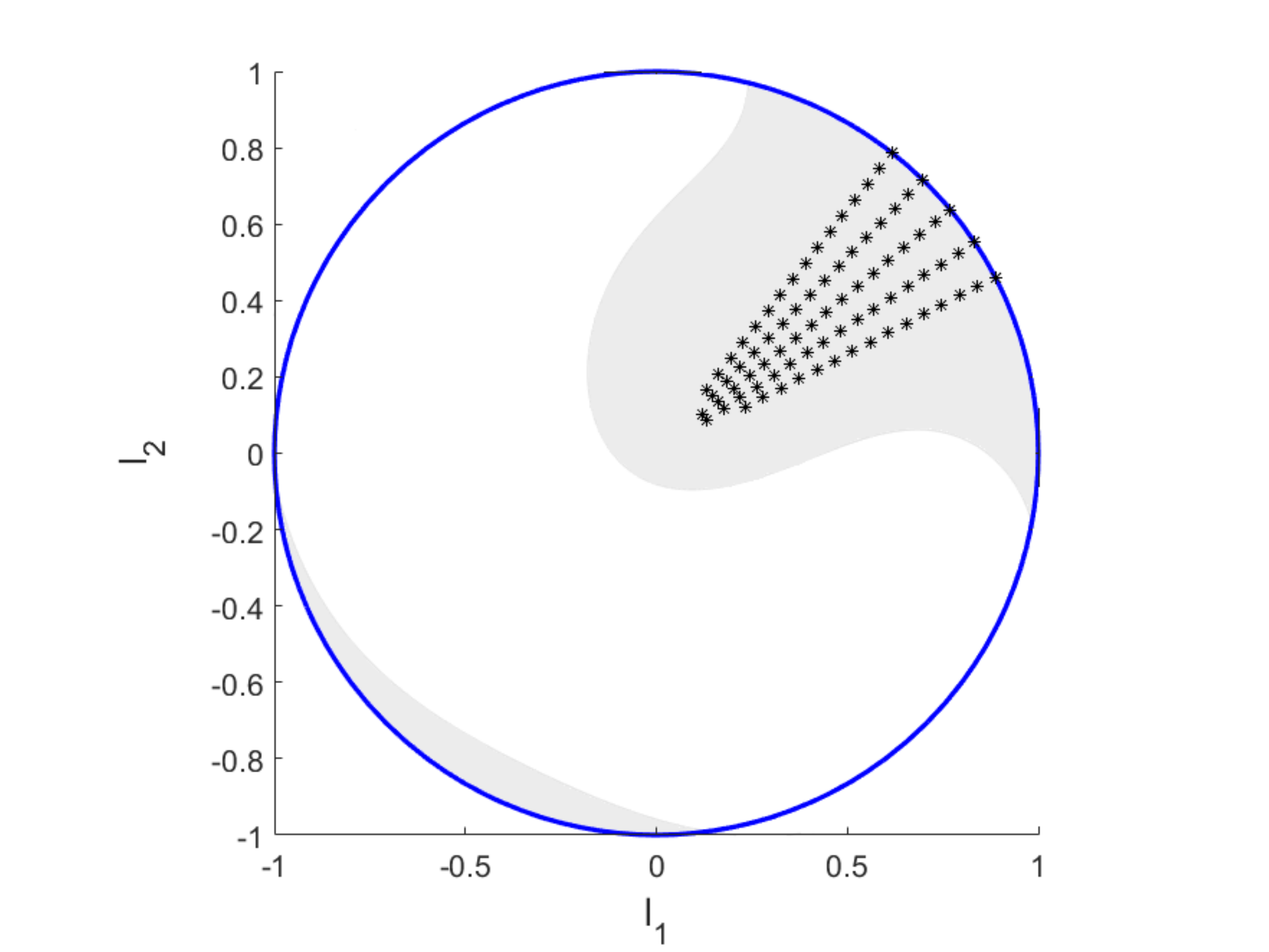}}
\subfigure{\includegraphics[width=\columnwidth,clip=true]{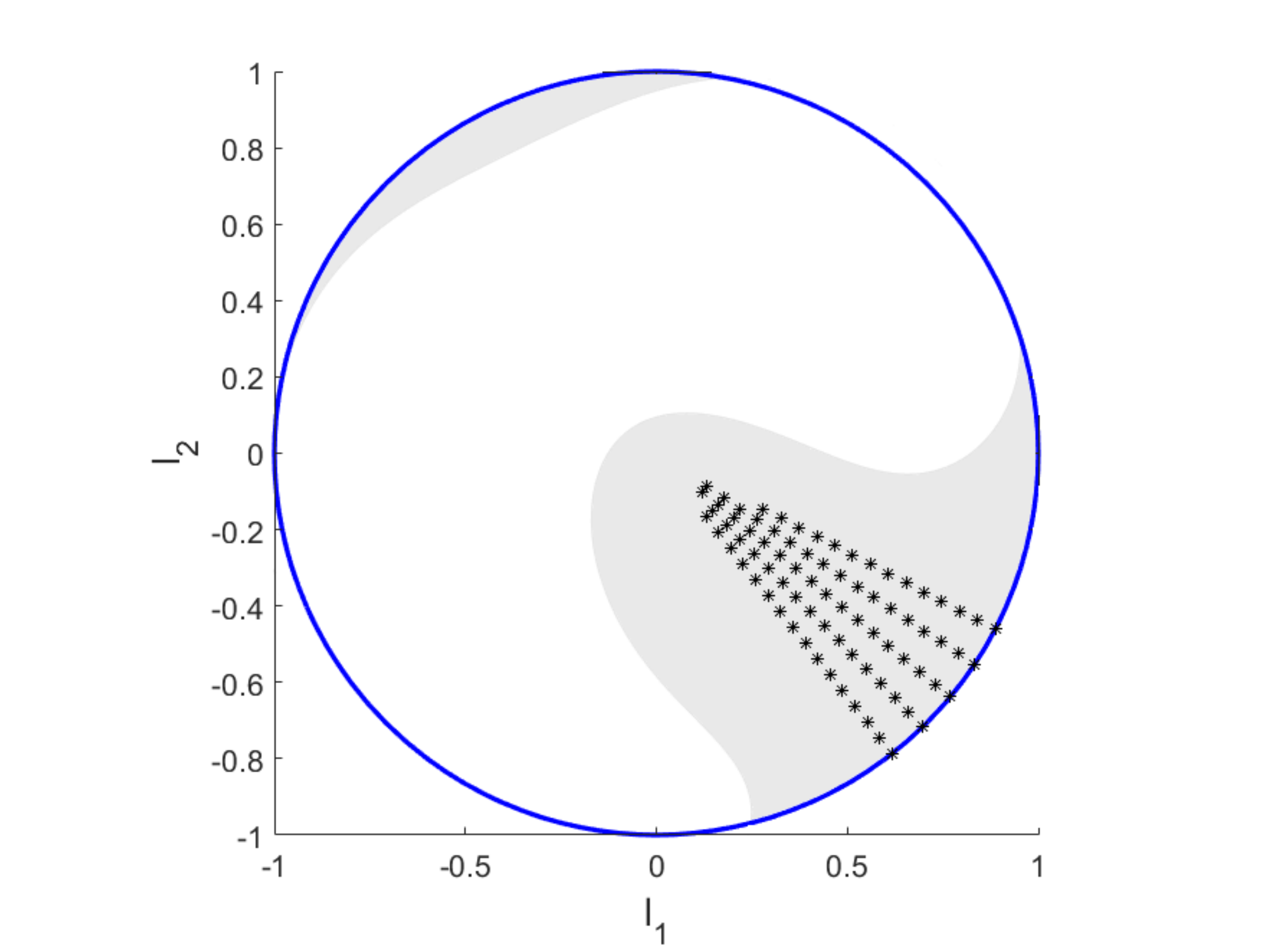}}
\caption{An illustration of slices of the computed $w_d$ for the nonlinear system described in Sec.~\ref{subsec:nonlinear} when $e_0 = [0.2; 0.2]$ (top) and when $e_0 = [0.2; -0.2]$ (bottom). The gray area inside L represents the 1-super level set of $w_d$. Dots, which are obtained by sampling, represent the sampled ground truth admissible l in each slice.}
\label{fig:eg2}
\vspace*{-.2in}
\end{figure}

\begin{figure}[!h]
\centering  
 {\includegraphics[width=\columnwidth,clip=true]{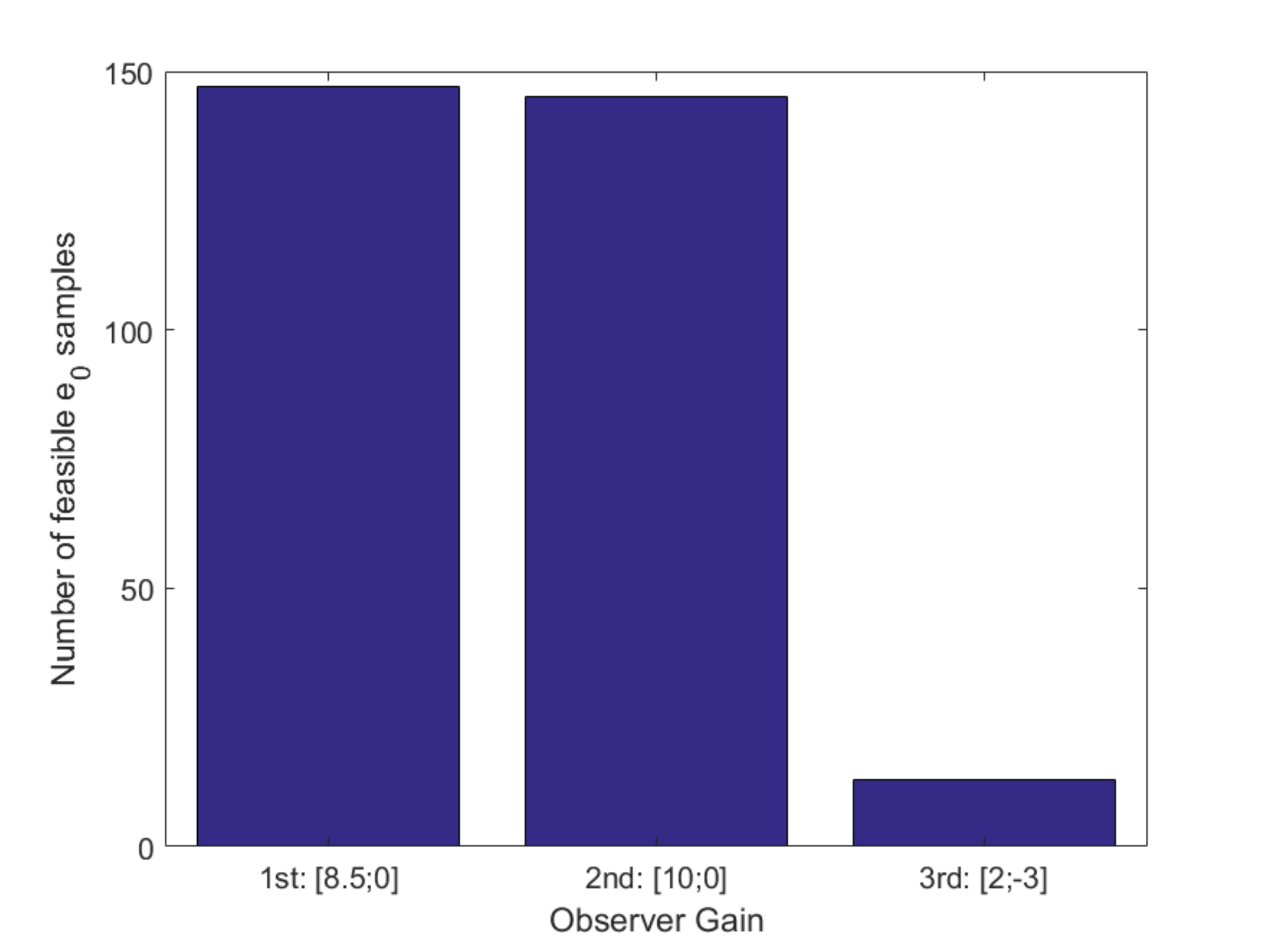}}
\caption{Bar chart depicting the number of admissible initial errors in $E$ for each associated gain. $l_{\text{sample}}$ (left) was generated by sampling, and $l_{\text{opt}}$ (middle) was generated by our proposed method. The last gain was chosen arbitrarily.}
\label{fig:eg2_hist}
\end{figure}
\section{Conclusion}
\label{sec:conclusion}
This paper describe a convex optimization technique to design an observer with static output injection for nonlinear systems.
By utilizing the notion of occupation measures, this paper proposes a two-step methodology to synthesize the gains that ensure the largest possible set of initial observer states converge to a state estimate with a desired estimation error in finite time regardless of the true initial state of the system being observed.
The first step optimizes over the space of polynomials using SDPs to find an outer approximation to the set of gains and associated initial estimation errors that have satisfactory estimation error. 
A similar framework can be applied to find an inner approximation to the set of adequate gains and initial error states.
The second step utilizes this set to select a gain that can drive the largest set of initial estimation errors to a suitable estimation error in finite time. 
The proposed method is validated numerically on several examples of varying complexities.

\bibliographystyle{ieeetr}
\bibliography{references}

\begin{thebibliography}{10}

\bibitem{perruquetti2008finite}
W.~Perruquetti, T.~Floquet, and E.~Moulay, ``Finite-time observers: application
  to secure communication,'' {\em IEEE Transactions on Automatic Control},
  vol.~53, no.~1, pp.~356--360, 2008.

\bibitem{asad2000slid}
A.~Azemi and E.~E. Yaz, ``Sliding-mode adaptive observer approach to chaotic
  synchronization,'' {\em ASME Journal of Dynamic Systems, Measurement, and
  Control}, vol.~122, no.~4, pp.~758--765, 2000.

\bibitem{du2013recursive}
H.~Du, C.~Qian, S.~Yang, and S.~Li, ``Recursive design of finite-time
  convergent observers for a class of time-varying nonlinear systems,'' {\em
  Automatica}, vol.~49, no.~2, pp.~601--609, 2013.

\bibitem{frye2010fast}
M.~Frye, S.~Ding, C.~Qian, and S.~Li, ``Fast convergent observer design for
  output feedback stabilisation of a planar vertical takeoff and landing
  aircraft,'' {\em IET control theory \& applications}, vol.~4, no.~4,
  pp.~690--700, 2010.

\bibitem{engel2002continuous}
R.~Engel and G.~Kreisselmeier, ``A continuous-time observer which converges in
  finite time,'' {\em IEEE Transactions on Automatic Control}, vol.~47, no.~7,
  pp.~1202--1204, 2002.

\bibitem{menold2003finite}
P.~H. Menold, R.~Findeisen, and F.~Allgower, ``Finite time convergent observers
  for nonlinear systems,'' in {\em Decision and Control, 2003. Proceedings.
  42nd IEEE Conference on}, vol.~6, pp.~5673--5678, IEEE, 2003.

\bibitem{allgower1999nonlinear}
F.~Allg{\"o}wer, T.~A. Badgwell, J.~S. Qin, J.~B. Rawlings, and S.~J. Wright,
  ``Nonlinear predictive control and moving horizon estimation-an introductory
  overview,'' in {\em Advances in control}, pp.~391--449, Springer, 1999.

\bibitem{xia1989nonlinear}
X.-H. Xia and W.-B. Gao, ``Nonlinear observer design by observer error
  linearization,'' {\em SIAM Journal on Control and Optimization}, vol.~27,
  no.~1, pp.~199--216, 1989.

\bibitem{krener1985nonlinear}
A.~J. Krener and W.~Respondek, ``Nonlinear observers with linearizable error
  dynamics,'' {\em SIAM Journal on Control and Optimization}, vol.~23, no.~2,
  pp.~197--216, 1985.

\bibitem{kazantzis1998nonlinear}
N.~Kazantzis and C.~Kravaris, ``Nonlinear observer design using lyapunov's
  auxiliary theorem,'' {\em Systems \& Control Letters}, vol.~34, no.~5,
  pp.~241--247, 1998.

\bibitem{folland2013real}
G.~B. Folland, {\em Real analysis: modern techniques and their applications}.
\newblock John Wiley \& Sons, 2013.

\bibitem{henrion2014convex}
D.~Henrion and M.~Korda, ``Convex computation of the region of attraction of
  polynomial control systems,'' {\em IEEE Transactions on Automatic Control},
  vol.~59, no.~2, pp.~297--312, 2014.

\bibitem{holmes2016convex}
P.~Holmes, S.~Kousik, S.~Mohan, and R.~Vasudevan, ``Convex estimation of the
  $\alpha$-confidence reachable set for systems with parametric uncertainty,''
  in {\em Decision and Control (CDC), 2016 IEEE 55th Conference on},
  pp.~4097--4103, IEEE, 2016.

\bibitem{mohan2016convex}
S.~Mohan, V.~Shia, and R.~Vasudevan, ``Convex computation of the reachable set
  for hybrid systems with parametric uncertainty,'' {\em arXiv preprint
  arXiv:1601.01019}, 2016.

\bibitem{Mohan2016}
S.~Mohan and R.~Vasudevan, ``Convex computation of the reachable set for hybrid
  systems with parametric uncertainty,'' in {\em 2016 American Control
  Conference (ACC)}, pp.~5141--5147, July 2016.

\bibitem{Anderson1987}
E.~Anderson and P.~Nash, {\em Linear programming in infinite-dimensional
  spaces: theory and applications}.
\newblock Wiley-Interscience series in discrete mathematics and optimization,
  Wiley, 1987.

\bibitem{lasserre2009moments}
J.~B. Lasserre, {\em Moments, positive polynomials and their applications},
  vol.~1.
\newblock World Scientific, 2009.

\bibitem{parrilo2000structured}
P.~A. Parrilo, {\em Structured semidefinite programs and semialgebraic geometry
  methods in robustness and optimization}.
\newblock PhD thesis, Citeseer, 2000.

\bibitem{korda2013inner}
M.~Korda, D.~Henrion, and C.~N. Jones, ``Inner approximations of the region of
  attraction for polynomial dynamical systems,'' {\em IFAC Proceedings
  Volumes}, vol.~46, no.~23, pp.~534--539, 2013.

\bibitem{lofberg2005yalmip}
J.~Lofberg, ``Yalmip: A toolbox for modeling and optimization in matlab,'' in
  {\em Computer Aided Control Systems Design, 2004 IEEE International Symposium
  on}, pp.~284--289, IEEE, 2005.

\end{thebibliography}
\end{document}